\documentclass{amsproc}
\usepackage{eurosym}
\usepackage{amsmath}
\usepackage{amsfonts}
\usepackage[none]{hyphenat}
\usepackage[hang]{footmisc}

\theoremstyle{plain}

\newtheorem*{acknowledgement*}{Acknowledgement}

\newtheorem{corollary}{Corollary}

\newtheorem{definition}{Definition}
\newtheorem{example}{Example}

\newtheorem{lemma}{Lemma}

\newtheorem{proposition}{Proposition}

\newtheorem{theorem}{Theorem}
\numberwithin{equation}{section}
\input{tcilatex}

\begin{document}
\title{\textbf{Clairaut conformal submersions from
Ricci solitons}}
\author{Murat Polat}
\date{}
\maketitle

\begin{abstract}
In the present article, we characterize Clairaut conformal submersions whose
total manifolds admit a Ricci soliton and provide a non-trivial example of
such Clairaut conformal submersions. We firstly calculate scalar curvature
and Ricci tensors of total manifolds of Clairaut conformal submersions and
provide necessary conditions for the fibres of such Clairaut conformal
submersions to be almost Ricci solitons and Einstein. Further, we provide
necessary conditions for the base manifold to be Ricci soliton and Einstein.
Then, we find a necessary condition for vector field $\dot{\zeta}$ to be
conformal vector field and killing vector field. Besides, we indicate that
if total manifolds of Clairaut conformal submersions admit a Ricci soliton
with the potential mean curvature vector field $H$ of $Ker\vartheta _{\ast
}, $ then the total manifolds of Clairaut conformal submersions admit a
gradient Ricci soliton. Finally, by solving Poisson equation, we acquire a
necessary and sufficient condition for Clairaut conformal submersions to be
harmonic.
\end{abstract}

\textbf{Keywords:} Ricci soliton, Clairaut conformal submersion, harmonic
map, Riemannian manifold, Riemannian submersion.

\textbf{AMS. 2020:} 53C12; 53B20; 53C43; 53C25.

\section{Introduction}

In \cite{Neill_1966}, O'Neill defined that conformal submersions are a
generalization of Riemannian submersions. Ishihara \cite{Ishihara_1979} and
Fugledge \cite{Fugledge_1978} defined horizontally conformal maps which have
applications in computer graphics and medical imaging (brain imaging).
Gundmundsson found the fundamental equations for the conformal submersions
in \cite{Gundmundsson_1992}. In \cite{Falcitelli_2008}, Falcitelli et al.
showed that under what conditions conformal submersions would become
Riemannian submersions. In \cite{Baird_2003}, Baird et al. indicated that a
smooth map $\vartheta :(M_{1}^{d_{1}},g_{M_{1}})\rightarrow
(M_{2}^{d_{2}},g_{M_{2}})$ is called weakly conformal at $p_{1}\in M_{1}$ if
there exists a number $\sigma ^{2}(p_{1})$ such that
\begin{equation}
g_{M_{2}}(\vartheta _{\ast }X_{1},\vartheta _{\ast }X_{2})=\sigma
^{2}(p_{1})g_{M_{1}}(X_{1},X_{2}),  \label{eqn1.1}
\end{equation}%
for$~X_{1},X_{2}\in \Gamma (T_{p_{1}}M_{1}).$

Let $(M_{1}^{d_{1}},g_{M_{1}})$ and $(M_{2}^{d_{2}},g_{M_{2}})$ be
Riemannian manifolds, then $\vartheta :(M_{1}^{d_{1}},g_{M_{1}})\rightarrow
(M_{2}^{d_{2}},g_{M_{2}})$ is called horizontally weakly conformal map at $%
p_{1}\in M_{1}$ if either $i.$ $\vartheta _{\ast p_{1}}=0$, or $ii.$ $%
\vartheta _{\ast p_{1}}$ maps the horizontal space $\mathcal{H}%
_{p_{1}}=(Ker\vartheta _{\ast p_{1}})^{\bot }$ conformally onto $%
T_{\vartheta (p_{1})}M_{2}$, i.e. $\vartheta _{\ast p_{1}}$ is satisfies (%
\ref{eqn1.1}) for $X_{1},X_{2}\in \mathcal{H}_{p_{1}}$ and is surjective. If
a node $p_{1}$ is of type $i.$ then $p_{1}\in M_{1}$ is said to be critical
node of $\vartheta $ and if node $p_{1}$ is of type $ii.$ then $p_{1}\in
M_{1}$ is said to be regular node. The number $\sigma ^{2}(p_{1})$
(non-negative) is called the square dilation (square root $\sigma (p_{1})$
is called the dilation). If the gradient of its dilation $\sigma $ is
vertical, i.e. $\mathcal{H}(grad~\sigma )=0$ then a horizontally weakly
conformal map $\vartheta :(M_{1}^{d_{1}},g_{M_{1}})\rightarrow
(M_{2}^{d_{2}},g_{M_{2}})$ is called horizontally homothetic. If $\vartheta $
has no critical nodes then horizontally weakly conformal map $\vartheta $ is
called horizontally conformal submersion (HCS) \cite{Baird_2003}. Therefore,
a Riemannian submersion is a HCS with dilation identically one.

If $r$ is the distance to the axis of a surface of revolution and $\theta $
is the angle between a meridian and the velocity vector of a geodesic, then
Clairaut's relation \cite{doCarmo} states that $r\sin \theta $ is constant.
In \cite{Bishop_1972}, Bishop gave Clairaut submersions and provided a
necessary and sufficient condition for Riemannian submersions to be Clairaut
Riemannian submersions. Further the generalization of Clairaut Riemannian
submersions as Clairaut Riemannian maps were introduced in \cite{Sahin_2017a}
and \cite{Yadav_Clairaut Riem_base}. Recently, the notion of Clairaut
conformal submersion (CCS) is introduced by Meena and Zawadzki in \cite%
{Meena_2022}. The authors related the geometry of Clairaut Riemannian
submersions with the CCS through conformal changes in the metric.

In \cite{Hamilton_1988}, Hamilton came up with the idea of Ricci soliton to
obtain a desired metric on a Riemannian manifold $(M_{1},g_{M_{1}})$. The
Ricci flow is an evolution equation (it can also be called the heat
equation) and can be formulated as%
\begin{equation*}
\frac{\partial }{\partial s}g_{M_{1}}(s)=-2Ric.
\end{equation*}%
Furthermore, we can say that Ricci solitons are a generalization of an
Einstein metric and the self-similar solutions of Ricci flow are Ricci
solitons \cite{Besse_1987}. $(M_{1},g_{M_{1}},\xi ,\mu )$ is called Ricci
soliton if there exists a potential vector field (PVF) $\xi $ which
satisfies
\begin{equation}
\frac{1}{2}(L_{\xi }g_{M_{1}})+Ric+\mu g_{M_{1}}=0,  \label{eqn1.2}
\end{equation}%
where $L_{\xi }g_{M_{1}}$ indicates the Lie derivative of $g_{M_{1}}$, $Ric$
indicates the Ricci tensor of Riemannian manifold and $\mu $ indicates a
constant. If $\mu <0,$ $\mu =0$ or $\mu >0$, then $(M_{1},g_{M_{1}},\xi ,\mu
)$ is said to be shrinking, steady or expanding respectively. Furthermore, $%
\xi $ is said to be conformal vector field \cite{Deshmukh_2014} if it
provides $L_{\xi }g_{M_{1}}=2\beta g_{M_{1}}$, here $\beta $ is the
potential function of PVF.

In \cite{Perleman_2002}, Perelman showed that Ricci solitons are very useful
to solve the Poincar\'{e} conjecture. In \cite{Pigola_2011}, Pigola et al.,
by considering $\mu $ as a variable function, gave almost Ricci soliton. In
this way, the geometry of the Ricci soliton became the most important
subject of research because of its wide applications in theoretical physics
and its geometric significance.

Lately, Riemannian submersions whose total manifolds admit Ricci soliton,
almost Ricci soliton, almost $\eta $-Ricci-Bourguignon soliton, $\eta $%
-Ricci soliton, $\eta $-Ricci-Yamabe soliton, almost Yamabe soliton,
conformal $\eta $-Ricci soliton were investigated in \cite{Meric_2019}, \cite%
{Bejan_2021}, \cite{Chaubev_2022}, \cite{Fatima_2021}, \cite{Meric_2020},
\cite{Siddiqi_2020} and \cite{Siddiqui_2022}. Also, Riemannian maps whose
total or base manifolds admit a Ricci soliton were investigated in \cite%
{Yadav_Riem_total}, \cite{Yadav_Clairaut Riem_total}, \cite{GGupta_2022},
\cite{Yadav_Riem_base} and \cite{Yadav_Clairaut Riem_base}. Recently, in
\cite{Meena_2023}, conformal submersions whose total manifolds admitting a
Ricci soliton studied by Meena and Yadav. Also, Polat studied Clairaut
pointwise slant Submersion from locally product Riemannian manifolds in \cite%
{polat}. We investigate CCS whose total manifolds admitting a Ricci soliton
in this current article.

The article is organized as: in Sect. \ref{sec2}, for the conformal
submersion and CCS which are required for this article, we gather some basic
concepts. In Sect. \ref{sec3}, we obtain the scalar curvature and Ricci
tensor by using Ricci soliton. Then we provide necessary conditions for the
fibers of CCS $\vartheta $ to be almost Ricci soliton and Einstein.
Moreover, we provide necessary conditions for the base manifold to be Ricci
soliton and Einstein. Also, we discuss the harmonicity of CCS between
Riemannian manifolds whose total space is Ricci soliton. Finally, we prepare
a illustrative example to support the correctness of the theory of this
article.

\section{Preliminaries}

\label{sec2}

In this part, we remind a brief review of some basic facts for the notion of
submersion, Riemannian submersion, conformal submersion and CCS between
Riemannian manifolds.

Let $\vartheta :(M_{1}^{d_{1}},g_{M_{1}})\rightarrow
(M_{2}^{d_{2}},g_{M_{2}})$ be a map between Riemannian manifolds $%
(M_{1}^{d_{1}},g_{M_{1}})$ and $(M_{2}^{d_{2}},g_{M_{2}}).$ Then $\vartheta $
is called $C^{\infty }$-submersion if it has maximal rank at every node $%
p_{1}\in M_{1}$. The fibers of $\vartheta $ is defined as $\vartheta
^{-1}(p_{2}),$ for any $p_{2}\in M_{2}$. The vectors tangent to fibers build
the smooth vertical distribution denoting $\nu _{p_{1}}=Ker\vartheta _{\ast
p_{1}}$ and its orthogonal complementary is called horizontal distribution
denoting $\mathcal{H}_{p_{1}}$. $\nu $ and $\mathcal{H}$ indicate
projections on the vertical and horizontal distributions, respectively. A
vector field $D$ on $M_{1}$ is said to be projectable if there exists a
vector field $\tilde{D}$ on $M_{2}$ such that $\vartheta _{\ast p_{1}}(D)=%
\tilde{D}_{\vartheta (p_{1})}$. Then $D$ and $\tilde{D}$ are called $%
\vartheta $-related. So, $D$ is called the horizontal lift of $\tilde{D}$.

Let $\vartheta :(M_{1}^{d_{1}},g_{M_{1}})\rightarrow
(M_{2}^{d_{2}},g_{M_{2}})$ be a submersion between Riemannian manifolds $%
(M_{1}^{d_{1}},g_{M_{1}})$ and $(M_{2}^{d_{2}},g_{M_{2}}).$ We say that $%
\vartheta $ is a Riemannian submersion if $\vartheta _{\ast p_{1}}$
preserves the length of the horizontal vectors.

Let $\vartheta :(M_{1}^{d_{1}},g_{M_{1}})\rightarrow
(M_{2}^{d_{2}},g_{M_{2}})$ be a Riemannian submersion. If $\vartheta _{\ast
} $ restricted to horizontal distribution of $\vartheta $ is a conformal
map, namely,
\begin{equation}
g_{M_{2}}(\vartheta _{\ast }X_{1},\vartheta _{\ast }X_{2})=\sigma
^{2}(p_{1})g_{M_{1}}(X_{1},X_{2}),  \label{eqn2.1}
\end{equation}%
$\forall X_{1},X_{2}\in \Gamma (Ker\vartheta _{\ast })^{\bot }~$and$%
~p_{1}\in M_{1}$ for a smooth function $\sigma :M_{1}\rightarrow \mathbb{R}%
^{+},$ then $\vartheta $ is said to be conformal submersion. The O'Neill
tensors $S$ and $T$ formulated as \cite{Neill_1966}
\begin{equation}
S_{E_{1}}E_{2}=\mathcal{H}\nabla _{\mathcal{H}E_{1}}\nu E_{2}+\nu \nabla _{%
\mathcal{H}E_{1}}\mathcal{H}E_{2},  \label{eqn2.2}
\end{equation}%
\begin{equation}
T_{E_{1}}E_{2}=\mathcal{H}\nabla _{\nu E_{1}}\nu E_{2}+\nu \nabla _{\nu
E_{1}}\mathcal{H}E_{2},  \label{eqn2.3}
\end{equation}%
for any $E_{1},E_{2}\in \Gamma (TM_{1})$, here $\nabla $ denotes the
Levi-Civita connection of metric tensor $g_{M_{1}}$. $\forall E_{1}\in
\Gamma (TM_{1})$, $T_{E_{1}}$ and $S_{E_{1}}$ are skew-symmetric operators
on $(\Gamma (TM_{1}),g_{M_{1}}).$ We can easily see that $T_{E_{1}}=T_{\nu
E_{1}}$ i.e, $T$ is vertical, and $S_{E_{1}}=S_{\mathcal{H}E_{1}}$ i.e, $S$
is horizontal and $T$ satisfies $T_{U_{1}}U_{2}=T_{U_{2}}U_{1},~\forall
U_{1},U_{2}\in \Gamma (Ker\vartheta _{\ast })$. Now, from (\ref{eqn2.2}) and
(\ref{eqn2.3}), we have
\begin{equation}
\nabla _{U_{1}}U_{2}=T_{U_{1}}U_{2}+\nu \nabla _{U_{1}}U_{2},  \label{eqn2.4}
\end{equation}%
\begin{equation}
\nabla _{X_{1}}U_{1}=S_{X_{1}}U_{1}+\nu \nabla _{X_{1}}U_{1},  \label{eqn2.5}
\end{equation}%
\begin{equation}
\nabla _{X_{1}}X_{2}=S_{X_{1}}X_{2}+\mathcal{H}\nabla _{X_{1}}X_{2},
\label{eqn2.6}
\end{equation}%
where $U_{1},U_{2}\in \Gamma (Ker\vartheta _{\ast })$ and $X_{1},X_{2}\in
\Gamma (Ker\vartheta _{\ast })^{\bot }$. A conformal submersion $\vartheta $
with
\begin{equation}
T_{U_{1}}U_{2}=g_{M_{1}}(U_{1},U_{2})H~\text{or}%
~T_{U_{1}}X_{1}=-g_{M_{1}}(H,X_{1})U_{1},  \label{eqn2.7}
\end{equation}%
is called as totally umbilical fibers \cite{Zawadzki_2014, Zawadzki_2020}.
Herein, $\forall U_{1},U_{2}\in \Gamma (Ker\vartheta _{\ast })$, $X_{1}\in
\Gamma (Ker\vartheta _{\ast })^{\bot }$ and $H$ denotes the mean curvature
vector field of the fibers.

\begin{proposition}
\label{prop2.1} \cite{Gundmundsson_1992} Let $\vartheta
:(M_{1},g_{M_{1}})\rightarrow (M_{2},g_{M_{2}})$ be a HCS such that $%
(Ker\vartheta _{\ast })^{\bot }$ is integrable. Then the horizontal space is
totally umbilical in $(M_{1},g_{M_{1}})$, i.e. $%
S_{X_{1}}X_{2}=g_{M_{1}}(X_{1},X_{2})H^{\prime }~\forall X_{1},X_{2}\in
\Gamma (Ker\vartheta _{\ast })^{\bot }$, where $H^{\prime }$ is the mean
curvature vector field of $(Ker\vartheta _{\ast })^{\bot }$ and stated
\begin{equation}
H^{\prime }=\left( \nabla _{\nu }\frac{1}{\sigma ^{2}}\right) -\frac{\sigma
^{2}}{2}.  \label{eqn2.9}
\end{equation}
\end{proposition}

The second fundamental form of $\vartheta $ \cite{Nore_1966} is given by $%
(\nabla \vartheta _{\ast })(X_{1},X_{2})={\nabla }_{X_{1}}^{\vartheta
}\vartheta _{\ast }X_{2}-\vartheta _{\ast }({\nabla }%
_{X_{1}}^{M_{1}}X_{2}),~\forall X_{1},X_{2}\in \Gamma (TM_{1})$, or
\begin{equation}
(\nabla \vartheta _{\ast })(X_{1},X_{2})={\nabla }_{\vartheta _{\ast
}X_{1}}^{M_{2}}\vartheta _{\ast }X_{2}-\vartheta _{\ast }({\nabla }%
_{X_{1}}^{M_{1}}X_{2}),~\forall X_{1},X_{2}\in \Gamma (TM_{1}).
\label{eqn2.10}
\end{equation}

\begin{lemma}
\label{lem2.1} \cite{Gundmundsson_1992} Let $\vartheta
:(M_{1}^{d_{1}},g_{M_{1}})\rightarrow (M_{2}^{d_{2}},g_{M_{2}})$ be a HCS.
Then
\begin{eqnarray*}
&&\vartheta _{\ast }(\mathcal{H}\nabla _{X_{1}}X_{2}) \\
&=&\frac{\sigma ^{2}}{2}\left\{ -g_{M_{1}}(X_{1},X_{2})\vartheta _{\ast
}(grad_{\mathcal{H}}~(\frac{1}{\sigma ^{2}}))+X_{1}(\frac{1}{\sigma ^{2}}%
)\vartheta _{\ast }X_{2}+X_{2}(\frac{1}{\sigma ^{2}})\vartheta _{\ast
}X_{1}\right\} \\
&&+\nabla _{\vartheta _{\ast }X_{1}}^{M_{2}}\vartheta _{\ast }X_{2},
\end{eqnarray*}%
where $\nabla $ Levi-Civita connection on $M_{1}$ and $X_{1},X_{2}$ are
basic vector fields.
\end{lemma}

Now we give general information about the gradient $(\func{grad})$,
divergence $(\func{div})$ and Hessian $(Hess)$ \cite{Sahin_2017}. Let $\beta
\in $\QTR{cal}{F}$(M_{1})$, then gradient of function $\beta $, indicated by
$grad\beta $ or $\nabla \beta $, stated
\begin{equation}
g_{M_{1}}(grad~\beta ,X_{1})=X_{1}(\beta ),  \label{eqn2.17}
\end{equation}%
for $X_{1}\in \Gamma (TM_{1}).$ Let $\{e_{k}\}_{1\leq k\leq d_{1}}$ be an
orthonormal basis of $T_{p_{1}}M_{1}$ then, we have
\begin{equation}
g_{M_{1}}(X_{1},X_{2})=%
\sum_{k=1}^{d_{1}}g_{M_{1}}(X_{1},e_{k})g_{M_{1}}(X_{2},e_{k}).
\label{eqn2.18}
\end{equation}%
The divergence of $X_{1}$ given by
\begin{equation}
div(X_{1})=\sum_{k=1}^{d_{1}}g_{M_{1}}(\nabla _{e_{k}}X_{1},e_{k}),
\label{eqn2.19}
\end{equation}%
for any $X_{1}\in \Gamma (TM_{1}).$ The Hessian tensor $h_{\beta }$ of $%
\beta $ is given by
\begin{equation*}
h_{\beta }(X_{1})=\nabla _{X_{1}}\nabla \beta ,
\end{equation*}%
for $X_{1}\in \Gamma (TM_{1}).$ The Hessian form of $\beta $ is given by
\begin{equation}
Hess\beta (X_{1},X_{2})=g_{M_{1}}(h_{\beta }(X_{1}),X_{2}),  \label{eqn2.21}
\end{equation}%
for any $X_{1},X_{2}\in \Gamma (TM_{1}).$ The Laplacian of $\beta $ is given
by
\begin{equation}
\Delta \beta =div(\nabla \beta ).  \label{eqn2.22}
\end{equation}

\begin{lemma}
\label{lem2.2} \cite{Gundmundsson_1992} Let $\beta :M_{1}\rightarrow \mathbb{%
R}$ be a smooth function and $(M_{1},g_{M_{1}})$ be a Riemannian manifold.
Then
\begin{equation*}
g_{M_{1}}(\nabla _{X_{1}}grad(\beta ),X_{2})=g_{M_{1}}(\nabla
_{X_{2}}grad(\beta ),X_{1}),
\end{equation*}

for$~X_{1},X_{2}\in \Gamma (TM_{1}).$
\end{lemma}

\begin{proposition}
\label{prp1}\cite{Gundmundsson_1992} Let $\vartheta
:(M_{1},g_{M_{1}})\rightarrow (M_{2},g_{M_{2}})$ be a HCS with dilation $%
\sigma $. Then
\begin{equation*}
S_{X_{1}}X_{2}=\frac{1}{2}\left\{ \nu \lbrack X_{1},X_{2}]-\sigma
^{2}g_{M_{1}}(X_{1},X_{2})\left( \nabla _{\nu }\frac{1}{\sigma ^{2}}\right)
\right\} ,
\end{equation*}

for any $X_{1},X_{2}\in \Gamma (Ker\vartheta _{\ast })^{\bot }.$
\end{proposition}

\begin{theorem}
\label{thr1}\cite{Meena_2023} Let $\vartheta :(M_{1},g_{M_{1}})\rightarrow
(M_{2},g_{M_{2}})$ be a HCS with dilation $\sigma $ such that $(Ker\vartheta
_{\ast })^{\bot }$ is totally geodesic. Then $\sigma $ is constant on $%
Ker\vartheta _{\ast }$.
\end{theorem}

\begin{theorem}
\label{thr2}\cite{Meena_2023} Let $\vartheta :(M_{1},g_{M_{1}})\rightarrow
(M_{2},g_{M_{2}})$ be a HCS. If $S$ is parallel then $\sigma $ is constant
on $Ker\vartheta _{\ast }$.
\end{theorem}

\begin{theorem}
\label{thr3}\cite{Meena_2023} Let $\vartheta :(M_{1},g_{M_{1}})\rightarrow
(M_{2},g_{M_{2}})$ be a HCS. Then $\vartheta $ is totally geodesic if and
only if fibers of $\vartheta $ are totally geodesic, $(Ker\vartheta _{\ast
})^{\bot }$ is totally geodesic and $\vartheta $ is homothetic.
\end{theorem}

\begin{lemma}
\label{lem2.3} \cite{Meena_2023} Let $\vartheta
:(M_{1}^{d_{1}},g_{M_{1}})\rightarrow (M_{2}^{d_{2}},g_{M_{2}})$ be a HCS
with dilation $\sigma $ such that $(Ker\vartheta _{\ast })^{\bot }$ is
integrable. Then followings are valid:

1. $\sum%
\limits_{l=1}^{d_{2}}g_{M_{1}}(S_{X_{l}}U_{1},S_{X_{l}}U_{2})=d_{2}^{2}\frac{%
\sigma ^{4}}{4}g_{M_{1}}(\nabla _{\nu }\frac{1}{\sigma ^{2}}%
,U_{1})g_{M_{1}}(\nabla _{\nu }\frac{1}{\sigma ^{2}},U_{2})$,

2. $\sum\limits_{l=1}^{d_{2}}g_{M_{1}}((\nabla
_{U_{1}}S)_{X_{l}}X_{l},U_{2})=d_{2}g_{M_{1}}(\nabla _{U_{1}}H^{\prime
},U_{2})$,

3. $\sum\limits_{l=1}^{d_{2}}g_{M_{1}}((\nabla
_{X_{1}}S)_{X_{l}}X_{l},U_{1})=d_{2}g_{M_{1}}(\nabla _{X_{1}}H^{\prime
},U_{1})$,

4. $\sum\limits_{l=1}^{d_{2}}g_{M_{1}}((\nabla
_{X_{l}}S)_{X_{1}}X_{l},U_{1})=g_{M_{1}}(X_{1},X_{l})g_{M_{1}}(\nabla
_{X_{l}}H^{\prime },U_{1})$,

5. $\sum\limits_{k=d_{2}+1}^{d_{1}}g_{M_{1}}((\nabla
_{U_{k}}S)_{X_{1}}X_{2},U_{k})=dkv(H^{\prime })g_{M_{1}}(X_{1},X_{2})$,

6. $\sum\limits_{k=d_{2}+1}^{d_{1}}g_{M_{1}}(S_{X_{1}}U_{k},S_{X_{2}}U_{k})=%
\frac{\sigma ^{4}}{4}|\nabla _{\nu }\frac{1}{\sigma ^{2}}%
|^{2}g_{M_{1}}(X_{1},X_{2})$,

here $\{U_{k}\}_{d_{2}+1\leq k\leq d_{1}}$ is orthonormal bases of $%
Ker\vartheta _{\ast }$ and $\{X_{l}\}_{1\leq l\leq d_{2}}$ is orthonormal
bases of $(Ker\vartheta _{\ast })^{\bot },$ for $U_{1},U_{2}\in \Gamma
(Ker\vartheta _{\ast })~$and$~X_{1},X_{2}\in \Gamma (Ker\vartheta _{\ast
})^{\bot }$.
\end{lemma}

\begin{proposition}
Let $\vartheta :(M_{1}^{d_{1}},g_{M_{1}})\rightarrow
(M_{2}^{d_{2}},g_{M_{2}})$ be a HCS with dilation $\sigma $. Then%
\begin{eqnarray}
&&Ric(U_{1},U_{2})  \label{Ric(U,V)} \\
&=&Ric^{\nu }(U_{1},U_{2})-(d_{1}-d_{2})g_{M_{1}}(T_{U_{1}}U_{2},H)  \notag
\\
&&+\sum\limits_{l=1}^{d_{2}}g_{M_{1}}(S_{X_{l}}U_{1},S_{X_{l}}U_{2})+\sum%
\limits_{l=1}^{d_{2}}g_{M_{1}}((\nabla _{U_{1}}S)_{X_{l}}X_{l},U_{2})  \notag
\\
&&-\sum\limits_{l=1}^{d_{2}}g_{M_{1}}((\nabla _{X_{l}}T)_{U_{1}}X_{l},U_{2})
\notag \\
&&-\frac{\sigma ^{4}}{2}d_{2}g_{M_{1}}(U_{1},\nabla _{\nu }\frac{1}{\sigma
^{2}})g_{M_{1}}(U_{2},\nabla _{\nu }\frac{1}{\sigma ^{2}}),  \notag
\end{eqnarray}%
\begin{eqnarray}
&&Ric(U_{1},X_{1})  \label{Ric(U,X)} \\
&=&(d_{1}-d_{2})g_{M_{1}}(\nabla _{U_{1}}H,X_{1})  \notag \\
&&-\sum\limits_{k=d_{2}+1}^{d_{1}}g_{M_{1}}((\nabla
_{U_{k}}T)_{U_{1}}U_{k},X_{1})+\sum\limits_{l=1}^{d_{2}}g_{M_{1}}((\nabla
_{X_{1}}S)_{X_{l}}X_{l},U_{1})  \notag \\
&&-\sum\limits_{l=1}^{d_{2}}g_{M_{1}}((\nabla
_{X_{l}}S)_{X_{1}}X_{l},U_{1})-\sum%
\limits_{l=1}^{d_{2}}g_{M_{1}}(T_{U_{1}}X_{l},\nu \lbrack X_{1},X_{l}]),
\notag
\end{eqnarray}%
\begin{eqnarray}
&&Ric(X_{1},X_{2})  \label{Ric(X,Y)} \\
&=&\sum\limits_{k=d_{2}+1}^{d_{1}}g_{M_{1}}((\nabla
_{U_{k}}S)_{X_{1}}X_{2},U_{k})  \notag \\
&&-\sum\limits_{k=d_{2}+1}^{d_{1}}g_{M_{1}}((\nabla
_{X_{1}}T)_{U_{k}}X_{2},U_{k})+\sum%
\limits_{k=d_{2}+1}^{d_{1}}g_{M_{1}}(S_{X_{1}}U_{k},S_{X_{2}}U_{k})  \notag
\\
&&-\sum\limits_{k=d_{2}+1}^{d_{1}}g_{M_{1}}(T_{U_{k}}X_{1},T_{U_{k}}X_{2})+%
\sigma ^{2}g_{M_{1}}(S_{X_{1}}X_{2},\nabla _{\nu }\frac{1}{\sigma ^{2}})
\notag \\
&&+\frac{1}{\sigma ^{2}}Ric^{M_{2}}(\tilde{X}_{1},\tilde{X}_{2})+\frac{3}{4}%
\sum\limits_{l=1}^{d_{2}}g_{M_{1}}(\nu \lbrack X_{1},X_{l}],\nu \lbrack
X_{l},X_{2}])  \notag \\
&&-\frac{(d_{2}-2)}{2}\sigma ^{2}g_{M_{1}}(\nabla _{X_{1}}\nabla \frac{1}{%
\sigma ^{2}},X_{2})  \notag \\
&&-\frac{\sigma ^{2}}{2}g_{M_{1}}(X_{1},X_{2})\left\{ \Delta ^{\mathcal{H}}%
\frac{1}{\sigma ^{2}}-d_{2}\left( H^{\prime }\frac{1}{\sigma ^{2}}\right)
\right\}  \notag \\
&&+\frac{d_{2}\sigma ^{4}}{4}g_{M_{1}}(X_{1},X_{2})|\nabla \frac{1}{\sigma
^{2}}|^{2}+\frac{\sigma ^{4}}{4}(d_{2}-2)(X_{1}\frac{1}{\sigma ^{2}})(X_{2}%
\frac{1}{\sigma ^{2}}),  \notag
\end{eqnarray}%
here $\{U_{k}\}_{d_{2}+1\leq k\leq d_{1}}$ is orthonormal bases of $%
Ker\vartheta _{\ast }$ and $\{X_{l}\}_{1\leq l\leq d_{2}}$ is orthonormal
bases of $(Ker\vartheta _{\ast })^{\bot },$ for $U_{1},U_{2}\in \Gamma
(Ker\vartheta _{\ast })~$and$~X_{1},X_{2}\in \Gamma (Ker\vartheta _{\ast
})^{\bot }$. Also, $X_{1}$ and $X_{2}$ are the horizontal lift of $\tilde{X}%
_{1}$ and $\tilde{X}_{2}$, respectively.
\end{proposition}

The Clairaut condition for a Riemannian submersion was originally defined by
Bishop in \cite{Bishop_1972}. Further, Clairaut condition was defined for a
Riemannian map in \cite{Sahin_2017a} and \cite{Yadav_Clairaut Riem_base}.
Now, we recall the definition CCS \cite{Meena_2022}. According to
definition, A conformal submersion $\vartheta :(M_{1},g_{M_{1}})\rightarrow
(M_{2},g_{M_{2}})$ between Riemannian manifolds with dilation $\sigma $ is
said to be CCS if there is a function $r:M_{1}\rightarrow \mathbb{R}^{+}$
such that for every geodesic $\zeta $ on $M_{1}$, the function $(r\circ
\zeta )\sin \omega (s)$ is constant along $\zeta $, where for all $s$, $%
\omega (s)$ is the angle between $\dot{\zeta}(s)$ and the horizontal space
at $\zeta (s)$.

\begin{theorem}
\label{Clairaut_condition} \cite{Meena_2022} Let $\vartheta
:(M_{1},g_{M_{1}})\rightarrow (M_{2},g_{M_{2}})$ be a conformal submersion
with connected fibers and dilation $\sigma $. Then $\vartheta $ is a CCS
with $r=e^{\beta }$ if and only if $\nabla \beta $ is horizontal, fibers of $%
\vartheta $ are totally umbilical with $H=-\nabla \beta $ i.e. $%
T_{U_{1}}U_{2}=-g_{M_{1}}(U_{1},U_{2})\nabla \beta $ and $\sigma $ is
constant along the fibers of $\vartheta $.
\end{theorem}

\section{Clairaut conformal submersions whose total manifolds admit a Ricci
soliton}

\label{sec3}In this part, we conceive a CCS whose total space is a Ricci
soliton and study its geometry.

\begin{proposition}
\label{prop3.1} Let $\vartheta :(M_{1}^{d_{1}},g_{M_{1}})\rightarrow
(M_{2}^{d_{2}},g_{M_{2}})$ be a CCS with dilation $\sigma $. Then the Ricci
tensor
\begin{eqnarray}
&&Ric(U_{1},U_{2})  \label{eqn3.1} \\
&=&Ric^{\nu }(U_{1},U_{2})-(d_{1}-d_{2})g_{M_{1}}(U_{1},U_{2})|\nabla \beta
|^{2}-g_{M_{1}}(U_{1},U_{2})div(\nabla \beta ),  \notag
\end{eqnarray}%
\begin{eqnarray}
&&Ric(U_{1},X_{1})  \label{eqn3.2} \\
&=&(d_{1}-d_{2}+1)g_{M_{1}}(S_{X_{1}}U_{1},\nabla \beta
)-\sum\limits_{l=1}^{d_{2}}\sum\limits_{k=d_{2}+1}^{d_{1}}g_{M_{1}}(\nabla
_{X_{l}}S_{X_{1}}X_{l},U_{k}),  \notag
\end{eqnarray}%
\begin{eqnarray}
&&Ric(X_{1},X_{2})  \label{eqn3.3} \\
&=&div(S_{X_{1}}X_{2})-2\sum%
\limits_{k=d_{2}+1}^{d_{1}}g_{M_{1}}(S_{X_{1}}U_{k},S_{X_{2}}U_{k})  \notag
\\
&&-(d_{1}-d_{2})g_{M_{1}}(X_{2},\nabla _{X_{1}}\nabla \beta )-X_{1}(\beta
)X_{2}(\beta )(d_{1}-d_{2})  \notag \\
&&+\frac{1}{\sigma ^{2}}Ric^{M_{2}}(\tilde{X}_{1},\tilde{X}_{2})-\frac{%
(d_{2}-2)}{2}\sigma ^{2}g_{M_{1}}(\nabla _{X_{1}}\nabla \frac{1}{\sigma ^{2}}%
,X_{2})  \notag \\
&&-\frac{\sigma ^{2}}{2}g_{M_{1}}(X_{1},X_{2})\Delta ^{\mathcal{H}}\frac{1}{%
\sigma ^{2}}+\frac{d_{2}\sigma ^{4}}{4}g_{M_{1}}(X_{1},X_{2})|\nabla \frac{1%
}{\sigma ^{2}}|^{2}  \notag \\
&&+\frac{\sigma ^{4}}{4}(d_{2}-2)(X_{1}\frac{1}{\sigma ^{2}})(X_{2}\frac{1}{%
\sigma ^{2}}),  \notag
\end{eqnarray}%
here $\{U_{k}\}_{d_{2}+1\leq k\leq d_{1}}$ is orthonormal bases of $%
Ker\vartheta _{\ast }$ and $\{X_{l}\}_{1\leq l\leq d_{2}}$ is orthonormal
bases of $(Ker\vartheta _{\ast })^{\bot },$ for $U_{1},U_{2}\in \Gamma
(Ker\vartheta _{\ast })~$and$~X_{1},X_{2}\in \Gamma (Ker\vartheta _{\ast
})^{\bot }$. Also, $X_{1}$ and $X_{2}$ are the horizontal lift of $\tilde{X}%
_{1}$ and $\tilde{X}_{2}$, respectively.
\end{proposition}

\begin{proof}
Since $\vartheta $ is CCS, utilizing Theorem \ref{Clairaut_condition}, we
obtain
\begin{equation}
g_{M_{1}}(T_{U_{1}}U_{2},H)=g_{M_{1}}(U_{1},U_{2})|\nabla \beta |^{2},
\label{eqn3.4}
\end{equation}%
and
\begin{equation}
\sum\limits_{l=1}^{d_{2}}g_{M_{1}}((\nabla
_{X_{l}}T)_{U_{1}}X_{l},U_{2})=g_{M_{1}}(U_{1},U_{2})div(\nabla \beta ).
\label{eqn3.5}
\end{equation}%
Also, by using Theorem \ref{Clairaut_condition} in Lemma \ref{lem2.3}, we
have
\begin{equation}
\sum\limits_{l=1}^{d_{2}}g_{M_{1}}(S_{X_{l}}U_{1},S_{X_{l}}U_{2})=0,
\label{eqn3.6}
\end{equation}%
and
\begin{equation}
\sum\limits_{l=1}^{d_{2}}g_{M_{1}}((\nabla _{U_{1}}S)_{X_{l}}X_{l},U_{2})=0.
\label{eqn3.7}
\end{equation}%
Then using (\ref{eqn3.4})-(\ref{eqn3.7}) in (\ref{Ric(U,V)}), we get (\ref%
{eqn3.1}).

Also, by using Theorem \ref{Clairaut_condition}, we obtain
\begin{equation}
g_{M_{1}}(\nabla _{U_{1}}H,X_{1})=-g_{M_{1}}(\nabla _{U_{1}}\nabla \beta
,X_{1})  \label{eqn3.8}
\end{equation}%
\begin{equation}
\sum\limits_{k=d_{2}+1}^{d_{1}}g_{M_{1}}((\nabla
_{U_{k}}T)_{U_{1}}U_{k},X_{1})=g_{M_{1}}(S_{X_{1}}U_{1},\nabla \beta )
\label{eqn3.9}
\end{equation}%
\begin{equation}
\sum\limits_{l=1}^{d_{2}}g_{M_{1}}((\nabla _{X_{1}}S)_{X_{l}}X_{l},U_{1})=0
\label{eqn3.10}
\end{equation}%
\begin{eqnarray}
&&\sum\limits_{l=1}^{d_{2}}g_{M_{1}}((\nabla _{X_{l}}S)_{X_{1}}X_{l},U_{1})
\label{eqn3.11} \\
&=&\sum\limits_{l=1}^{d_{2}}\sum\limits_{k=d_{2}+1}^{d_{1}}g_{M_{1}}((\nabla
_{X_{l}}S)_{X_{1}}X_{l},U_{k})g_{M_{1}}(U_{k},U_{k})  \notag
\end{eqnarray}%
\begin{equation}
\sum\limits_{l=1}^{d_{2}}g_{M_{1}}(T_{U_{1}}X_{l},\nu \lbrack
X_{1},X_{l}])=-2g_{M_{1}}(S_{X_{1}}U_{1},\nabla \beta ).  \label{eqn3.12}
\end{equation}%
Then using (\ref{eqn3.8})-(\ref{eqn3.12}) in (\ref{Ric(U,X)}), we get (\ref%
{eqn3.2}). Finally,
\begin{equation}
\sum\limits_{k=d_{2}+1}^{d_{1}}g_{M_{1}}((\nabla
_{U_{k}}S)_{X_{1}}X_{2},U_{k})=div(S_{X_{1}}X_{2})  \label{eqn3.13}
\end{equation}%
\begin{equation}
\sum\limits_{k=d_{2}+1}^{d_{1}}g_{M_{1}}(S_{X_{1}}U_{k},S_{X_{2}}U_{k})=0
\label{eqn3.14}
\end{equation}%
\begin{equation}
\sum\limits_{k=d_{2}+1}^{d_{1}}g_{M_{1}}((\nabla
_{X_{1}}T)_{U_{k}}X_{2},U_{k})=(d_{1}-d_{2})g_{M_{1}}(X_{2},\nabla
_{X_{1}}\nabla \beta )  \label{eqn3.15}
\end{equation}%
\begin{equation}
\sum%
\limits_{k=d_{2}+1}^{d_{1}}g_{M_{1}}(T_{U_{k}}X_{1},T_{U_{k}}X_{2})=g_{M_{1}}\left( \nabla _{X_{1}}%
\mathcal{H}\nabla \frac{1}{\sigma ^{2}},X_{2}\right) .  \label{eqn3.16}
\end{equation}%
Then using (\ref{eqn3.13})-(\ref{eqn3.16}) in (\ref{Ric(X,Y)}), we get (\ref%
{eqn3.3}).
\end{proof}

\begin{theorem}
Let $(M_{1},g_{M_{1}},\xi ,\mu )$ be a Ricci soliton with the PVF $\xi \in
\Gamma (TM_{1})$ and $\vartheta :(M_{1}^{d_{1}},g_{M_{1}})\rightarrow
(M_{2}^{d_{2}},g_{M_{2}})$ be a CCS between Riemannian manifolds. Then
following options are valid:

(i) any fiber of $\vartheta $ is an almost Ricci soliton if $\xi =W\in
\Gamma (Ker\vartheta _{\ast })$,

(ii) any fiber of $\vartheta $ is an Einstein if $\xi =X_{1}\in \Gamma
(Ker\vartheta _{\ast })^{\bot }$.
\end{theorem}

\begin{proof}
Since $(M_{1},g_{M_{1}},\xi ,\mu )$ is a Ricci soliton, then using (\ref%
{eqn1.2}), we obtain
\begin{equation}
\frac{1}{2}(L_{\xi }g_{M_{1}})(U_{1},U_{2})+Ric(U_{1},U_{2})+\mu
g_{M_{1}}(U_{1},U_{2})=0,~  \label{eqn3.4a}
\end{equation}%
for$~U_{1},U_{2}\in \Gamma (Ker\vartheta _{\ast }).$ Using (\ref{eqn3.1}) in
(\ref{eqn3.4a}), we have
\begin{eqnarray}
&&\frac{1}{2}(L_{\xi }g_{M_{1}})(U_{1},U_{2})+Ric^{\nu }(U_{1},U_{2})
\label{eqn3.5a} \\
&&+g_{M_{1}}(U_{1},U_{2})(\mu -(d_{1}-d_{2})|\nabla \beta |^{2}-div(\nabla
\beta ))  \notag \\
&=&0.  \notag
\end{eqnarray}%
If $\xi =W\in \Gamma (Ker\vartheta _{\ast })$, then from (\ref{eqn3.5a}), we
acquire
\begin{equation*}
\frac{1}{2}(L_{W}g_{M_{1}})(U_{1},U_{2})+Ric^{\nu }(U_{1},U_{2})+\mu
^{\prime }g_{M_{1}}(U_{1},U_{2})=0,
\end{equation*}%
where $\mu ^{\prime 2}-div(\nabla \beta )$ is a smooth function on $M_{1}$.
This implies $(i)$. Now, if $\xi =X_{1}\in \Gamma (Ker\vartheta _{\ast
})^{\bot },$ then (\ref{eqn3.5a}) implies
\begin{eqnarray*}
&&\frac{1}{2}\{g_{M_{1}}(\nabla _{U_{1}}X_{1},U_{2})+g_{M_{1}}(\nabla
_{U_{2}}X_{1},U_{1})\} \\
&&+Ric^{\nu }(U_{1},U_{2})+g_{M_{1}}(U_{1},U_{2})(\mu -(d_{1}-d_{2})|\nabla
\beta |^{2}-div(\nabla \beta )) \\
&=&0,
\end{eqnarray*}%
which implies
\begin{eqnarray*}
&&-g_{M_{1}}(\mathcal{H}\nabla _{U_{1}}U_{2},X_{1})+Ric^{\nu }(U_{1},U_{2})
\\
&&+g_{M_{1}}(U_{1},U_{2})(\mu -(d_{1}-d_{2})|\nabla \beta |^{2}-div(\nabla
\beta )) \\
&=&0.
\end{eqnarray*}%
Using (\ref{eqn2.4}) and Theorem \ref{Clairaut_condition} in above equation,
we get
\begin{equation*}
Ric^{\nu }(U_{1},U_{2})+g_{M_{1}}(U_{1},U_{2})(\mu -(d_{1}-d_{2})|\nabla
\beta |^{2}-div(H)+X_{1}(\beta ))=0.
\end{equation*}%
which implies $(ii)$.
\end{proof}

\begin{theorem}
Let $(M_{1},g_{M_{1}},\xi ,\mu )$ be a Ricci soliton with the PVF $\xi \in
\Gamma (TM_{1})$ and $\vartheta :(M_{1}^{d_{1}},g_{M_{1}})\rightarrow
(M_{2}^{d_{2}},g_{M_{2}})$ be a totally geodesic CCS between Riemannian
manifolds. Then the following options are valid:

(i) $M_{2}$ is a Ricci soliton if $\xi =Z\in \Gamma (Ker\vartheta _{\ast
})^{\bot }$,

(ii) $M_{2}$ is an Einstein if $\xi =W\in \Gamma (Ker\vartheta _{\ast })$.
\end{theorem}

\begin{proof}
Since $(M_{1},g_{M_{1}},\xi ,\mu )$ is a Ricci soliton then using (\ref%
{eqn1.2}), we obtain%
\begin{equation*}
\frac{1}{2}(L_{\xi }g_{M_{1}})(X_{1},X_{2})+Ric(X_{1},X_{2})+\mu
g_{M_{1}}(X_{1},X_{2})=0,~
\end{equation*}

which can be written as%
\begin{eqnarray}
&&\frac{1}{2}\left\{ g_{M_{1}}(\nabla _{X_{1}}\xi ,X_{2})+g_{M_{1}}(\nabla
_{X_{2}}\xi ,X_{1})\right\}  \label{eqn3.19} \\
&&+Ric(X_{1},X_{2})+\mu g_{M_{1}}(X_{1},X_{2})  \notag \\
&=&0,  \notag
\end{eqnarray}

for$~X_{1},X_{2}\in \Gamma (Ker\vartheta _{\ast })^{\bot }.$ If $\xi =Z\in
\Gamma (Ker\vartheta _{\ast })^{\bot },$ then utilizing (\ref{eqn1.1}) in (%
\ref{eqn3.19})%
\begin{eqnarray}
&&\frac{1}{2\sigma ^{2}}\left\{ g_{M_{2}}(\vartheta _{\ast }(\nabla
_{X_{1}}Z),\vartheta _{\ast }X_{2})+g_{M_{2}}(\vartheta _{\ast }(\nabla
_{X_{2}}Z),\vartheta _{\ast }X_{1})\right\}  \label{eqn3.20} \\
&&+Ric(X_{1},X_{2})+\mu g_{M_{1}}(X_{1},X_{2})  \notag \\
&=&0.  \notag
\end{eqnarray}

Using (\ref{eqn2.9}) in (\ref{eqn3.20}), we have%
\begin{eqnarray}
&&\frac{1}{2\sigma ^{2}}\left\{
\begin{array}{c}
g_{M_{2}}(\nabla _{X_{1}}^{M_{2}}\bar{Z}-(\nabla \vartheta _{\ast
})(X_{1},Z),\tilde{X}_{2}) \\
+g_{M_{2}}(\nabla _{X_{2}}^{M_{2}}\bar{Z}-(\nabla \vartheta _{\ast
})(X_{2},Z),\bar{X}_{1})%
\end{array}%
\right\}  \label{eqn3.21} \\
&&+Ric(X_{1},X_{2})+\frac{\mu }{2\sigma ^{2}}g_{M_{2}}(\bar{X}_{1},\tilde{X}%
_{2})  \notag \\
&=&0,  \notag
\end{eqnarray}

where $X_{1},X_{2}$ and $Z$ are the horizontal lift of $\bar{X}_{1},\tilde{X}%
_{2}$ and $\bar{Z},$ respectively. Since $\vartheta $ is totally geodesic,
then $(\nabla \vartheta _{\ast })(X_{1},Z)=0$ and $(\nabla \vartheta _{\ast
})(X_{2},Z)=0.$ Using Theorem \ref{thr1}, Theorem \ref{thr2} and Theorem \ref%
{thr3} in (\ref{eqn3.3}) then (\ref{eqn3.21}) can be written as%
\begin{eqnarray*}
&&\frac{1}{2\sigma ^{2}}\left\{ g_{M_{2}}(\nabla _{X_{1}}^{M_{2}}\bar{Z},%
\tilde{X}_{2})+g_{M_{2}}(\nabla _{\tilde{X}_{2}}^{M_{2}}\bar{Z},\bar{X}%
_{1})\right\} \\
&&+\frac{1}{\sigma ^{2}}Ric^{M_{2}}(\bar{X}_{1},\tilde{X}_{2})+\frac{\mu }{%
\sigma ^{2}}g_{M_{2}}(\bar{X}_{1},\tilde{X}_{2}) \\
&=&0,
\end{eqnarray*}

which completes proof of (i).

If $\xi =W\in \Gamma (Ker\vartheta _{\ast })$ then (\ref{eqn3.19}) becomes
\begin{eqnarray*}
&&\frac{1}{2}\left\{ g_{M_{1}}(\nabla _{X_{1}}W,X_{2})+g_{M_{1}}(\nabla
_{X_{2}}W,X_{1})\right\} \\
&&+Ric(X_{1},X_{2})+\mu g_{M_{1}}(X_{1},X_{2}) \\
&=&0,
\end{eqnarray*}

which can be written as
\begin{eqnarray*}
&&-\frac{1}{2}\left\{ g_{M_{1}}(\nabla _{X_{1}}X_{2},W)+g_{M_{1}}(\nabla
_{X_{2}}X_{1},W)\right\} \\
&&+Ric(X_{1},X_{2})+\mu g_{M_{1}}(X_{1},X_{2}) \\
&=&0.
\end{eqnarray*}

Using (\ref{eqn2.1}) in (\ref{eqn2.6}), we have%
\begin{eqnarray}
&&-\frac{1}{2}\left\{
g_{M_{1}}(S_{X_{1}}X_{2},W)+g_{M_{1}}(S_{X_{2}}X_{1},W)\right\}
\label{eqn3.22} \\
&&+Ric(X_{1},X_{2})+\mu g_{M_{1}}(X_{1},X_{2})  \notag \\
&=&0.  \notag
\end{eqnarray}

Since $\vartheta $ is totally geodesic, then using (\ref{eqn3.3}),
Proposition \ref{prp1} and Theorem \ref{thr3} in (\ref{eqn3.22}), we obtain%
\begin{equation*}
\frac{1}{\sigma ^{2}}Ric^{M_{2}}(\bar{X}_{1},\tilde{X}_{2})+\frac{\mu }{%
\sigma ^{2}}g_{M_{2}}(\bar{X}_{1},\tilde{X}_{2})=0,
\end{equation*}

where $X_{1}$ and $X_{2}$ are the horizontal lift of $\bar{X}_{1}$ and $%
\tilde{X}_{2},$ respectively. This completes the proof of (ii).
\end{proof}

\begin{corollary}
Let $(M_{1},g_{M_{1}},\xi ,\mu )$ be a Ricci soliton with the PVF $W\in
\Gamma (Ker\vartheta _{\ast })$ and $\vartheta
:(M_{1}^{d_{1}},g_{M_{1}})\rightarrow (M_{2}^{d_{2}},g_{M_{2}})$ be a
totally geodesic CCS between Riemannian manifolds. If $Ric^{M_{2}}(\bar{X}%
_{1},\tilde{X}_{2})=-\mu g_{M_{2}}(\bar{X}_{1},\tilde{X}_{2}),$ then $W$ is
a killing vector field on $M_{2}$.
\end{corollary}

\begin{theorem}
Let $\zeta $ be a geodesic curve on $M_{1}$ and $(M_{1},g_{M_{1}},\dot{\zeta}%
,\mu )$ be a Ricci soliton with the PVF $\dot{\zeta}\in \Gamma (TM_{1}).$
Let $\vartheta :(M_{1}^{d_{1}},g_{M_{1}})\rightarrow
(M_{2}^{d_{2}},g_{M_{2}})$ be a CCS with $r=e^{\beta }$ between Riemannian
manifolds such that fibers of $\vartheta $ are Einstein. Then $\dot{\zeta}$
is conformal vector field on $Ker\vartheta _{\ast }.$
\end{theorem}

\begin{proof}
Since $(M_{1},g_{M_{1}},\dot{\zeta},\mu )$ is a Ricci soliton followings are
obtained:%
\begin{equation}
\frac{1}{2}(L_{\dot{\zeta}}g_{M_{1}})(U_{1},U_{2})+Ric(U_{1},U_{2})+\mu
g_{M_{1}}(U_{1},U_{2})=0,~  \label{eqn3.23}
\end{equation}

for$~U_{1},U_{2}\in \Gamma (Ker\vartheta _{\ast })$ and $\dot{\zeta}\in
\Gamma (TM_{1}).$ Using (\ref{eqn3.1}) in (\ref{eqn3.23}), we have%
\begin{eqnarray*}
&&\frac{1}{2}(L_{\dot{\zeta}%
}g_{M_{1}})(U_{1},U_{2})+Ric^{v}(U_{1},U_{2})-(d_{1}-d_{2})g_{M_{1}}(U_{1},U_{2})\left\Vert \nabla \beta \right\Vert ^{2}
\\
&&-g_{M_{1}}(U_{1},U_{2})\func{div}(\nabla \beta )+\mu g_{M_{1}}(U_{1},U_{2})
\\
&=&0.
\end{eqnarray*}

Since fibers of $\vartheta $ is Einstein, using $Ric^{v}(U_{1},U_{2})=-\mu
g_{M_{1}}(U_{1},U_{2})$ we get%
\begin{equation*}
\frac{1}{2}(L_{\dot{\zeta}}g_{M_{1}})(U_{1},U_{2})-\left\{
(d_{1}-d_{2})\left\Vert \nabla \beta \right\Vert ^{2}+\func{div}(\nabla
\beta )\right\} g_{M_{1}}(U_{1},U_{2})=0,
\end{equation*}

which can be written as
\begin{equation*}
(L_{\dot{\zeta}}g_{M_{1}})(U_{1},U_{2})=2\beta _{1}g_{M_{1}}(U_{1},U_{2}),
\end{equation*}

where $\beta _{1}=(d_{1}-d_{2})\left\Vert \nabla \beta \right\Vert ^{2}+%
\func{div}(\nabla \beta )$ is a smooth function on $M_{1}.$ Thus $\dot{\zeta}
$ is conformal vector field on $Ker\vartheta _{\ast }.$
\end{proof}

\begin{corollary}
Let $\zeta $ be a geodesic curve on $M_{1}$ and $(M_{1},g_{M_{1}},\dot{\zeta}%
,\mu )$ be a Ricci soliton with the PVF $\dot{\zeta}\in \Gamma (TM_{1}).$
Let $\vartheta :(M_{1}^{d_{1}},g_{M_{1}})\rightarrow
(M_{2}^{d_{2}},g_{M_{2}})$ be a CCS with $r=e^{\beta }$ between Riemannian
manifolds. If fibers of $\vartheta $ are totally geodesic then, $\dot{\zeta}$
is killing vector field on $Ker\vartheta _{\ast }.$
\end{corollary}

\begin{theorem}
Let $\zeta $ be a geodesic curve on $M_{1}$ and $(M_{1},g_{M_{1}},\dot{\zeta}%
,\mu )$ be a Ricci soliton with the PVF $\dot{\zeta}\in \Gamma (TM_{1}).$
Let $\vartheta :(M_{1}^{d_{1}},g_{M_{1}})\rightarrow
(M_{2}^{d_{2}},g_{M_{2}})$ be a homothetic CCS with $r=e^{\beta }$ from a
Riemannian manifold $M_{1}$ to an Einstein manifold $M_{2}$ such that fibers
of $\vartheta $ are totally geodesic and $(Ker\vartheta _{\ast })^{\bot }$
is integrable. Then $\dot{\zeta}$ is conformal vector field on $%
(Ker\vartheta _{\ast })^{\bot }.$
\end{theorem}

\begin{proof}
Since $(M_{1},g_{M_{1}},\dot{\zeta},\mu )$ is a Ricci soliton then by (\ref%
{eqn1.2}) we get%
\begin{equation*}
\frac{1}{2}(L_{\dot{\zeta}}g_{M_{1}})(X_{1},X_{2})+Ric(X_{1},X_{2})+\mu
g_{M_{1}}(X_{1},X_{2})=0,~
\end{equation*}

for$~X_{1},X_{2}\in \Gamma (Ker\vartheta _{\ast })^{\bot }.$ Using Theorem %
\ref{thr1}, Theorem \ref{thr2} and Theorem \ref{thr3} in (\ref{eqn3.3}) then
we have%
\begin{eqnarray*}
&&\frac{1}{2}(L_{\dot{\zeta}}g_{M_{1}})(X_{1},X_{2})+\frac{1}{\sigma ^{2}}%
Ric^{M_{2}}(\bar{X}_{1},\tilde{X}_{2})-\frac{\sigma ^{2}}{2}%
g_{M_{1}}(X_{1},X_{2})\Delta ^{\mathcal{H}}\frac{1}{\sigma ^{2}} \\
&&+\frac{d_{2}\sigma ^{4}}{4}g_{M_{1}}(X_{1},X_{2})|\nabla \frac{1}{\sigma
^{2}}|^{2}+\mu g_{M_{1}}(X_{1},X_{2}) \\
&=&0.
\end{eqnarray*}

Using (\ref{eqn2.1}) in the above equation%
\begin{eqnarray*}
&&\frac{1}{2}(L_{\dot{\zeta}}g_{M_{1}})(X_{1},X_{2})+\frac{1}{\sigma ^{2}}%
Ric^{M_{2}}(\bar{X}_{1},\tilde{X}_{2})-\frac{1}{2}g_{M_{2}}(\bar{X}_{1},%
\tilde{X}_{2})\Delta ^{\mathcal{H}}\frac{1}{\sigma ^{2}} \\
&&+\frac{d_{2}\sigma ^{4}}{4}g_{M_{2}}(\bar{X}_{1},\tilde{X}_{2})|\nabla
\frac{1}{\sigma ^{2}}|^{2}+\frac{\mu }{\sigma ^{2}}g_{M_{2}}(\bar{X}_{1},%
\tilde{X}_{2}) \\
&=&0,
\end{eqnarray*}

which can be written as%
\begin{equation*}
\frac{1}{2}(L_{\dot{\zeta}}g_{M_{1}})(X_{1},X_{2})+\frac{1}{\sigma ^{2}}%
Ric^{M_{2}}(\bar{X}_{1},\tilde{X}_{2})+\beta _{1}g_{M_{2}}(\bar{X}_{1},%
\tilde{X}_{2})=0,
\end{equation*}

or
\begin{equation}
\frac{1}{2}(L_{\dot{\zeta}}g_{M_{1}})(X_{1},X_{2})+\frac{1}{\sigma ^{2}}%
\left\{ Ric^{M_{2}}(\bar{X}_{1},\tilde{X}_{2})+\beta _{2}g_{M_{2}}(\bar{X}%
_{1},\tilde{X}_{2})\right\} =0,  \label{eqn3.24}
\end{equation}%
where $\beta _{1}=-\frac{1}{2}\Delta ^{\mathcal{H}}\frac{1}{\sigma ^{2}}+%
\frac{d_{2}\sigma ^{4}}{4}|\nabla \frac{1}{\sigma ^{2}}|^{2}+\frac{\mu }{%
\sigma ^{2}}$ $\ $and $\beta _{2}=\sigma ^{2}\beta _{1}$ Since $M_{2}$ is an
Einstein, putting $Ric^{M_{2}}(\tilde{X}_{1},\tilde{X}_{2})=\beta
_{2}g_{M_{2}}(\tilde{X}_{1},\tilde{X}_{2})$ in (\ref{eqn3.24}), we obtain%
\begin{equation*}
\frac{1}{2}(L_{\dot{\zeta}}g_{M_{1}})(X_{1},X_{2})+2\beta
_{1}g_{M_{1}}(X_{1},X_{2})=0,
\end{equation*}

which implies that $\dot{\zeta}$ is conformal vector field on $(Ker\vartheta
_{\ast })^{\bot }$.
\end{proof}

\begin{theorem}
\label{Thr4} \cite{Meena_2023}Let $\vartheta
:(M_{1}^{d_{1}},g_{M_{1}})\rightarrow (M_{2}^{d_{2}},g_{M_{2}})$ be a
totally geodesic HCS with dilation $\sigma .$ Then
\begin{equation*}
s=s^{v}+\frac{1}{\sigma ^{2}}s^{M_{2}},
\end{equation*}

where $s,s^{v}$ and $s^{M_{2}}$ indicate the scalar curvatures of $%
M_{1},Ker\vartheta _{\ast }$ and $M_{2},$ respectively.
\end{theorem}

Using Theorem \ref{Thr4}, the following result can be written

\begin{theorem}
\label{Thr5}Let $(M_{1},g_{M_{1}},\dot{\zeta},\mu )$ be a Ricci soliton with
the PVF $\dot{\zeta}\in \Gamma (TM_{1}).$ and $\vartheta
:(M_{1}^{d_{1}},g_{M_{1}})\rightarrow (M_{2}^{d_{2}},g_{M_{2}})$ be a
totally geodesic CCS with $r=e^{\beta }$ between Riemannian manifolds such
that $(Ker\vartheta _{\ast })^{\bot }$ is integrable. Then $M_{1}$ has
constant scalar curvature by $-\mu d_{1}.$
\end{theorem}

\begin{proof}
Since $(M_{1},g_{M_{1}},\dot{\zeta},\mu )$ is a Ricci soliton then by (\ref%
{eqn1.2}) we have%
\begin{eqnarray}
&&\frac{1}{2}(g_{M_{1}}(\nabla _{X_{1}}\dot{\zeta},X_{2})+g_{M_{1}}(\nabla
_{X_{2}}\dot{\zeta},X_{1}))  \label{eqn3.25} \\
&&+Ric(X_{1},X_{2})+\mu g_{M_{1}}(X_{1},X_{2})  \notag \\
&=&0,  \notag
\end{eqnarray}

for $X_{1},X_{2},\dot{\zeta}\in \Gamma (TM_{1}).$ Now, we decompose $%
X_{1},X_{2}$ and $\dot{\zeta}$ such that $X_{1}=\nu X_{1}+\mathcal{H}%
X_{1},X_{2}=\nu X_{2}+\mathcal{H}X_{2}$ and $\dot{\zeta}=\nu W+\mathcal{H}Z$%
. Then (\ref{eqn3.25}) stated%
\begin{eqnarray}
&&\frac{1}{2}\left\{
\begin{array}{c}
g_{M_{1}}(\nabla _{\nu X_{1}+\mathcal{H}X_{1}}\nu W+\mathcal{H}Z,\nu X_{2}+%
\mathcal{H}X_{2}) \\
+g_{M_{1}}(\nabla _{\nu X_{2}+\mathcal{H}X_{2}}\nu W+\mathcal{H}Z,\nu X_{1}+%
\mathcal{H}X_{1})%
\end{array}%
\right\}  \label{eqn3.26} \\
&&+Ric(vX_{1},vX_{2})+Ric(\mathcal{H}X_{1},\mathcal{H}X_{2})+Ric(vX_{1},%
\mathcal{H}X_{2})  \notag \\
&&+Ric(\mathcal{H}X_{1},vX_{2})+\mu \left\{ g_{M_{1}}(vX_{1},vX_{2})+(%
\mathcal{H}X_{1},\mathcal{H}X_{2})\right\}  \notag \\
&=&0.  \notag
\end{eqnarray}

Taking trace of (\ref{eqn3.26}), we have
\begin{eqnarray}
&&\frac{1}{2}\left\{ 2\sum\limits_{l=1}^{d_{2}}g_{M_{1}}(\nabla
_{X_{l}}X_{l},X_{l})+2\sum\limits_{k=d_{2}+1}^{d_{1}}g_{M_{1}}(\nabla
_{Uk}U_{k},U_{k})\right\}  \label{eqn3.27} \\
&&+\sum\limits_{k=d_{2}+1}^{d_{1}}Ric(U_{k},U_{k})+\sum%
\limits_{l=1}^{d_{2}}Ric(X_{l},X_{l})+2\sum\limits_{k,l}Ric(U_{k},X_{l})
\notag \\
&&+\mu \left\{
\sum\limits_{l=1}^{d_{2}}g_{M_{1}}(X_{l},X_{l})+\sum%
\limits_{k=d_{2}+1}^{d_{1}}g_{M_{1}}(U_{k},U_{k})\right\}  \notag \\
&=&0,  \notag
\end{eqnarray}

where $\{U_{k}\}_{d_{2}+1\leq k\leq d_{1}}$ is orthonormal bases of $%
Ker\vartheta _{\ast }$ and $\{X_{l}\}_{1\leq l\leq d_{2}}$ is orthonormal
bases of $(Ker\vartheta _{\ast })^{\bot }$. Since $\vartheta $ is totally
geodesic and $(Ker\vartheta _{\ast })^{\bot }$ is integrable, then using
Theorem \ref{thr1}, Theorem \ref{thr2} and Theorem \ref{thr3} in (\ref%
{eqn3.1}), (\ref{eqn3.2}) and (\ref{eqn3.3}), we can write (\ref{eqn3.27}) as%
\begin{eqnarray}
&&\sum\limits_{k=d_{2}+1}^{d_{1}}g_{M_{1}}(\nabla
_{Uk}U_{k},U_{k})+\sum\limits_{l=1}^{d_{2}}g_{M_{1}}(\nabla
_{X_{l}}X_{l},X_{l})  \label{eqn3.28} \\
&&+Ric^{v}(U_{k},U_{k})+\frac{1}{\sigma ^{2}}Ric^{M_{2}}(\bar{X}_{1},\tilde{X%
}_{2})+\mu (d_{1}-d_{2}+d_{2})  \notag \\
&=&0,  \notag
\end{eqnarray}

where $Ric^{v}$ and $Ric^{M_{2}}$ indicate the scalar curvatures of $%
Ker\vartheta _{\ast }$ and $M_{2},$ respectively. From Theorem \ref{Thr4}, (%
\ref{eqn3.28}) and since $\nabla $ is metric connection on $M_{1},$ then we
have%
\begin{eqnarray*}
&&\frac{1}{2}\left\{ \sum\limits_{l=1}^{d_{2}}\nabla
_{X_{l}}g_{M_{1}}(X_{l},X_{l})+\sum\limits_{k=d_{2}+1}^{d_{1}}\nabla
_{U_{k}}g_{M_{1}}(U_{k},U_{k})\right\} \\
&&+Ric^{v}+\frac{1}{\sigma ^{2}}Ric^{M_{2}}+\mu d_{1} \\
&=&0.
\end{eqnarray*}

Thus we obtain $s+\mu d_{1}=0,$ where $s=Ric^{v}+\frac{1}{\sigma ^{2}}%
Ric^{M_{2}}$ is the scalar curvature on $M_{1}$.
\end{proof}

\begin{theorem}
Let $(M_{1},g_{M_{1}},-H,\mu )$ be a Ricci soliton with the PVF $-H\in
(Ker\vartheta _{\ast })^{\bot }$ and $\vartheta
:(M_{1}^{d_{1}},g_{M_{1}})\rightarrow (M_{2}^{d_{2}},g_{M_{2}})$ be a CCS
with $r=e^{\beta }$ between Riemannian manifolds. Then the following options
are true:

(i) $M_{1}$ admits a gradient Ricci soliton,

(ii) The mean curvature vector field of $Ker\vartheta _{\ast }$ is constant.
\end{theorem}

\begin{proof}
Since $(M_{1},g_{M_{1}},-H,\mu )$ is a Ricci soliton then using (\ref{eqn1.2}%
), we obtain%
\begin{eqnarray}
&&-\frac{1}{2}(g_{M_{1}}(\nabla _{X_{1}}H,X_{2})+g_{M_{1}}(\nabla
_{X_{2}}H,X_{1}))  \label{eqn3.29} \\
&&+Ric(X_{1},X_{2})+\mu g_{M_{1}}(X_{1},X_{2})  \notag \\
&=&0,  \notag
\end{eqnarray}

for $X_{1},X_{2}\in \Gamma (TM_{1})$ and $-H\in \Gamma (Ker\vartheta _{\ast
})^{\bot }.$ Using Theorem \ref{Clairaut_condition}, then (\ref{eqn3.29})
can be written as%
\begin{eqnarray*}
&&\frac{1}{2}(g_{M_{1}}(\nabla _{X_{1}}\nabla \beta ,X_{2})+g_{M_{1}}(\nabla
_{X_{2}}\nabla \beta ,X_{1})) \\
&&+Ric(X_{1},X_{2})+\mu g_{M_{1}}(X_{1},X_{2}) \\
&=&0,
\end{eqnarray*}

which is equal to%
\begin{eqnarray}
&&\frac{1}{2}(g_{M_{1}}(g_{M_{2}\beta
}(X_{1}),X_{2})+g_{M_{1}}(g_{M_{2}\beta }(X_{2}),X_{1}))  \label{eqn3.30} \\
&&+Ric(X_{1},X_{2})+\mu g_{M_{1}}(X_{1},X_{2})  \notag \\
&=&0.  \notag
\end{eqnarray}

Since $h_{\beta }$ is self-adjoint, (\ref{eqn3.30}) can be written as%
\begin{equation*}
g_{M_{1}}(h_{\beta }(X_{1}),X_{2})+Ric(X_{1},X_{2})+\mu
g_{M_{1}}(X_{1},X_{2})=0.
\end{equation*}

Using (\ref{eqn2.21}) in above equation%
\begin{equation}
Hess\beta (X_{1},X_{2})+Ric(X_{1},X_{2})+\mu g_{M_{1}}(X_{1},X_{2})=0,
\label{eqn3.31}
\end{equation}

which completes the proof of (i). If we taking trace (\ref{eqn3.31}), from (%
\ref{eqn2.22}) we have%
\begin{equation*}
\Delta \beta +s+\mu d_{1}=0.
\end{equation*}

Thus, the Poisson equation on $(M_{1},g_{M_{1}})$ is
\begin{equation}
\Delta \beta =\func{div}(\nabla \beta )=-s-\mu d_{1},  \label{eqn3.32}
\end{equation}

where $\Delta $ is the Laplace operator, $\beta $ is a smooth function on $%
M_{1}.$ To determine its solution using Theorem \ref{Thr5} in (\ref{eqn3.32}%
), we obtain%
\begin{equation*}
\Delta \beta =\func{div}(\nabla \beta )=0,i.e.\nabla (H)=0,
\end{equation*}

which means $H$ is constant. This completes the proof of (ii).
\end{proof}

\begin{definition}
\cite{Sahin_2017} Let $\vartheta :(M_{1}^{d_{1}},g_{M_{1}})\rightarrow
(M_{2}^{d_{2}},g_{M_{2}})$ be a smooth map between Riemannian manifolds.
Then $\vartheta $ is harmonic if and only if the tension field $\tau
(\vartheta )$ of $\vartheta $ vanishes at every node $p_{1}\in M_{1}.$
\end{definition}

Considering this definition, the following theorem can be written:

\begin{theorem}
Let $(M_{1},g_{M_{1}},W,\mu )$ be a Ricci soliton with the PVF $W\in \Gamma
(Ker\vartheta _{\ast })$ and $\vartheta
:(M_{1}^{d_{1}},g_{M_{1}})\rightarrow (M_{2}^{d_{2}},g_{M_{2}})$ be a
totally geodesic CCS with $r=e^{\beta }$ between Riemannian manifolds such
that $(Ker\vartheta _{\ast })^{\bot }$ is totally geodesic. Then $\vartheta $
is harmonic if and only if the scalar curvature of $Ker\vartheta _{\ast }$
is $-\mu (d_{1}-d_{2}).$
\end{theorem}

\begin{proof}
Since $\vartheta $ is a conformal submersion then we have%
\begin{equation}
\tau (\vartheta )=(d_{1}-d_{2})\vartheta _{\ast }(\nabla \beta )+(d_{2}-2)%
\frac{\sigma ^{2}}{2}\vartheta _{\ast }(\nabla _{\mathcal{H}}\frac{1}{\sigma
^{2}}).  \label{eqn3.33}
\end{equation}

Now, putting $\xi =W$ in (\ref{eqn3.5a}), we have%
\begin{eqnarray*}
&&\frac{1}{2}\left\{ g_{M_{1}}(\nabla _{U_{1}}W,U_{2})+g_{M_{1}}(\nabla
_{U_{2}}W,U_{1})\right\} +Ric^{\nu }(U_{1},U_{2}) \\
&&+g_{M_{1}}(U_{1},U_{2})\left\{ \mu -(d_{1}-d_{2})|\nabla \beta
|^{2}-div(\nabla \beta )\right\} \\
&=&0.
\end{eqnarray*}

Taking trace, we have%
\begin{eqnarray}
&&\sum\limits_{k=d_{2}+1}^{d_{1}}g_{M_{1}}(\nabla
_{U_{k}}U_{k},U_{k})+\sum\limits_{k=d_{2}+1}^{d_{1}}Ric^{v}(U_{k},U_{k})
\label{eqn3.34} \\
&&-\left\{ \mu -(d_{1}-d_{2})|\nabla \beta |^{2}-div(\nabla \beta )\right\}
\sum\limits_{k=d_{2}+1}^{d_{1}}g_{M_{1}}(U_{k},U_{k})  \notag \\
&=&0,  \notag
\end{eqnarray}

herein $\left\{ U_{k}\right\} _{d_{2}+1\leq k\leq d_{1}}$ is an orthonormal
bases of $Ker\vartheta _{\ast }.$ Since $\nabla $ is metric connection, then
(\ref{eqn3.34}) can be written as%
\begin{equation}
s^{Ker\vartheta _{\ast }}-\left\{ \mu -(d_{1}-d_{2})|\nabla \beta
|^{2}(d_{1}-d_{2})-div(\nabla \beta )\right\} =0,  \label{eqn3.35}
\end{equation}

where $s^{Ker\vartheta _{\ast }}$ is the scalar curvature of $Ker\vartheta
_{\ast }.$ Since $s^{Ker\vartheta _{\ast }}=-\mu (d_{1}-d_{2})$ then (\ref%
{eqn3.35}) implies%
\begin{equation}
(d_{1}-d_{2})^{2}|\nabla \beta |^{2}+div(\nabla \beta
)(d_{1}-d_{2})=0\Leftrightarrow \nabla \beta =0.  \label{eqn3.36}
\end{equation}

Thus from (\ref{eqn3.36}) and Lemma \ref{lem2.1} we obtain $\tau (\vartheta
)=0\Leftrightarrow \nabla \beta =0.$ This completes the proof.
\end{proof}

\begin{example}
Let $M_{1}=\left\{ (u_{1},u_{2},u_{3})\in
\mathbb{R}
^{3}:u_{1}\neq 0\right\} $ be a Riemannian manifold with Riemannian metric $%
g_{M_{1}}$ on $M_{1}$ stated $%
g_{M_{1}}=e^{-2u_{1}}du_{1}^{2}+e^{-2u_{1}}du_{2}^{2}+e^{-2u_{1}}du_{3}^{2}.$
Let $M_{2}=\left\{ (v_{1},v_{2})\in
\mathbb{R}
^{2}\right\} $ be a Riemannian manifold with Riemannian metric $g_{M_{2}}$
on $M_{2}$ stated $g_{M_{2}}=e^{2u_{1}}dv_{1}^{2}+e^{2u_{1}}dv_{2}^{2}.$
Define a map $\vartheta :(M_{1}^{d_{1}},g_{M_{1}})\rightarrow
(M_{2}^{d_{2}},g_{M_{2}})$ by%
\begin{equation*}
\vartheta (u_{1},u_{2},u_{3})=(u_{1},u_{2}).
\end{equation*}

by direct computations%
\begin{equation*}
Ker\vartheta _{\ast }=Span\left\{ V=e_{3}\right\} ,
\end{equation*}

and%
\begin{equation*}
(Ker\vartheta _{\ast })^{\bot }=Span\left\{ X_{1}=e_{1},X_{2}=e_{2}\right\} ,
\end{equation*}

herein $\left\{ e_{1}=e^{u_{1}}\frac{\partial }{\partial u_{1}}%
,e_{2}=e^{u_{1}}\frac{\partial }{\partial u_{2}},e_{3}=e^{u_{1}}\frac{%
\partial }{\partial u_{3}}\right\} $ are bases of $T_{p_{1}}M_{1}$ and $%
\left\{ e_{1}^{\ast }=e^{u_{1}}\frac{\partial }{\partial v_{1}},e_{2}^{\ast
}=e^{u_{1}}\frac{\partial }{\partial v_{2}}\right\} $ are bases of $%
T_{\vartheta (p_{1})}M_{2},$ for any $p_{1}\in M_{1}.$ By direct
computations, one can see that $\vartheta _{\ast }X_{1}=e_{1}^{\ast
},\vartheta _{\ast }X_{2}=e_{2}^{\ast },\vartheta _{\ast }V=0$ and $%
g_{M_{2}}(\vartheta _{\ast }X_{i},\vartheta _{\ast }X_{j})=\sigma
^{2}g_{M_{1}}(X_{i},X_{j}),$ for $\sigma =e^{2u_{1}}.$ Thus $\vartheta $ is
conformal submersion with dilation $\sigma =e^{u_{1}}.$ Now, we will find a
smooth function $\beta $ on $M_{1}$ satisfying $T_{V}V=-g_{M_{1}}(V,V)\nabla
\beta $ for $V\in \Gamma (Ker\vartheta _{\ast }).$ we can simply calculate
that%
\begin{eqnarray}
\Gamma _{11}^{1} &=&-1,\Gamma _{11}^{2}=\Gamma _{11}^{3}=0,  \label{eqn3.37}
\\
\Gamma _{21}^{2} &=&-1,\Gamma _{21}^{1}=\Gamma _{21}^{3}=0,  \notag \\
\Gamma _{12}^{2} &=&-1,\Gamma _{12}^{1}=\Gamma _{12}^{3}=0,  \notag \\
\Gamma _{22}^{1} &=&1,\Gamma _{22}^{2}=\Gamma _{22}^{3}=0,  \notag \\
\Gamma _{33}^{1} &=&1,\Gamma _{33}^{2}=\Gamma _{33}^{3}=0,  \notag \\
\Gamma _{31}^{3} &=&-1,\Gamma _{31}^{1}=\Gamma _{31}^{2}=0,  \notag \\
\Gamma _{13}^{3} &=&-1,\Gamma _{13}^{1}=\Gamma _{13}^{2}=0,  \notag \\
\Gamma _{32}^{1} &=&\Gamma _{32}^{2}=\Gamma _{32}^{3}=0,  \notag \\
\Gamma _{23}^{1} &=&\Gamma _{23}^{2}=\Gamma _{23}^{3}=0.  \notag
\end{eqnarray}

By using (\ref{eqn3.37}), we get%
\begin{eqnarray}
\nabla _{e_{1}}e_{1} &=&0,\nabla _{e_{2}}e_{2}=e^{2u_{1}}\frac{\partial }{%
\partial u_{1}},\nabla _{e_{3}}e_{3}=e^{2u_{1}}\frac{\partial }{\partial
u_{1}},  \label{eqn3.38} \\
\nabla _{e_{1}}e_{2} &=&0,\nabla _{e_{2}}e_{1}=-e^{2u_{1}}\frac{\partial }{%
\partial u_{2}},\nabla _{e_{3}}e_{1}=-e^{2u_{1}}\frac{\partial }{\partial
u_{3}},  \notag \\
\nabla _{e_{3}}e_{2} &=&\nabla _{e_{1}}e_{3}=\nabla _{e_{2}}e_{3}=0.  \notag
\end{eqnarray}

Using (\ref{eqn3.38}), we get%
\begin{equation*}
\nabla _{V}V=e^{u_{1}}X_{1}.
\end{equation*}

Its mean that $\mathcal{H}\nabla _{V}V=e^{u_{1}}X_{1},v\nabla _{V}V=0.$ By (%
\ref{eqn2.4}), we have%
\begin{equation*}
T_{V}V=e^{u_{1}}X_{1}.
\end{equation*}

Similarly, if we take $X=\lambda _{1}X_{1}+\lambda _{2}X_{2}$ for $\lambda
_{1},\lambda _{2}\in
\mathbb{R}
$ then%
\begin{equation*}
\nabla _{V}X=-\lambda _{1}e^{u_{1}}V.
\end{equation*}

Its mean that $\mathcal{H}\nabla _{V}X=0,v\nabla _{V}X=-\lambda
_{1}e^{u_{1}}V.$ Also,
\begin{equation*}
\nabla _{X}V=0,
\end{equation*}

i.e. $\mathcal{H}\nabla _{X}V=0,v\nabla _{X}V=0.$ Then by (\ref{eqn2.5}), we
have%
\begin{equation}
S_{X}V=0,  \label{eqn3.39}
\end{equation}

and
\begin{equation*}
\nabla _{X}X=\lambda _{2}^{2}e^{u_{1}}X_{1}-\lambda _{1}\lambda
_{2}e^{u_{1}}X_{2},
\end{equation*}

i.e. $\mathcal{H}\nabla _{X}V=\lambda _{2}^{2}e^{u_{1}}X_{1}-\lambda
_{1}\lambda _{2}e^{u_{1}}X_{2},v\nabla _{X}X=0.$ Then by (\ref{eqn2.6}), we
have%
\begin{equation}
S_{X}X=0.  \label{eqn3.40}
\end{equation}

For any smooth function on $M_{1},$ the gradient of $\beta $ with respect to
the metric $g_{M_{1}}$ is given by $\nabla \beta
=\sum\limits_{i,j=1}^{3}g_{M_{1}}^{ij}\frac{\partial \beta }{\partial u_{i}}%
\frac{\partial }{\partial u_{j}}.$ Therefore,%
\begin{equation*}
\nabla \beta =e^{2u_{1}}\frac{\partial \beta }{\partial u_{1}}\frac{\partial
}{\partial u_{1}}+e^{2u_{1}}\frac{\partial \beta }{\partial u_{2}}\frac{%
\partial }{\partial u_{2}}+e^{2u_{1}}\frac{\partial \beta }{\partial u_{3}}%
\frac{\partial }{\partial u_{3}}.
\end{equation*}

Hence,%
\begin{eqnarray}
\nabla \beta &=&-e^{2u_{1}}\frac{\partial }{\partial u_{1}}  \label{eqn3.41}
\\
&=&-e^{u_{1}}X_{1},  \notag
\end{eqnarray}

for the function $\beta =-u_{1}.$ On the other hand, we have%
\begin{equation*}
g_{M_{1}}(V,V)=1.
\end{equation*}

Then it is easy to verify that $T_{V}V=-g_{M_{1}}(V,V)\nabla \beta .$ Thus,
by (\ref{eqn2.7}) and Theorem \ref{Clairaut_condition}, we see that this map
is CCS with $r=e^{-u_{1}}.$ Now, we will show that $M_{1}$ admits a Ricci
soliton, i.e.%
\begin{equation}
\frac{1}{2}(L_{E}g_{M_{1}})(F,G)+Ric(F,G)+\mu g_{M_{1}}(F,G)=0,
\label{eqn3.42}
\end{equation}

for any $E,F,G\in \Gamma (TM_{1}).$ Since here the dimension of $%
Ker\vartheta _{\ast }$ is one and the dimension of $(Ker\vartheta _{\ast
})^{\bot }$ is two, therefore we can decompose $E,F$ and $G$ such that $%
F=\lambda _{1}V+\lambda _{2}X_{1}+\lambda _{3}X_{2},G=\lambda _{4}V+\lambda
_{5}X_{1}+\lambda _{6}X_{2}$ and $E=\lambda _{7}V+\lambda _{8}X_{1}+\lambda
_{9}X_{2},$ where $V$ denotes for component of $Ker\vartheta _{\ast }$ and $%
X_{1},X_{2}$ denote for component of $(Ker\vartheta _{\ast })^{\bot }$ and $%
\left\{ \lambda _{i}\right\} _{1\leq i\leq 9}\in
\mathbb{R}
$ are some scalars. Now, since
\begin{equation*}
\frac{1}{2}(L_{E}g_{M_{1}})(F,G)=\frac{1}{2}\left\{ g_{M_{1}}(\nabla
_{F}E,G)+g_{M_{1}}(\nabla _{G}E,F)\right\} ,
\end{equation*}

which is equal to%
\begin{eqnarray*}
&&\frac{1}{2}(L_{E}g_{M_{1}})(F,G) \\
&=&\frac{1}{2}\left\{
\begin{array}{c}
g_{M_{1}}(\nabla _{\lambda _{1}V+\lambda _{2}X_{1}+\lambda _{3}X_{2}}\lambda
_{7}V+\lambda _{8}X_{1}+\lambda _{9}X_{2},\lambda _{4}V+\lambda
_{5}X_{1}+\lambda _{6}X_{2}) \\
+g_{M_{1}}(\nabla _{\lambda _{4}V+\lambda _{5}X_{1}+\lambda
_{6}X_{2}}\lambda _{7}V+\lambda _{8}X_{1}+\lambda _{9}X_{2},\lambda
_{1}V+\lambda _{2}X_{1}+\lambda _{3}X_{2})%
\end{array}%
\right\} ,
\end{eqnarray*}

which implies%
\begin{equation}
\frac{1}{2}(L_{E}g_{M_{1}})(F,G)=\frac{1}{2}\left\{
\begin{array}{c}
e^{u_{1}}(\lambda _{1}\lambda _{5}\lambda _{7}-2\lambda _{1}\lambda
_{4}\lambda _{8} \\
-2\lambda _{3}\lambda _{6}\lambda _{8}+\lambda _{3}\lambda _{5}\lambda
_{9}+\lambda _{2}\lambda _{4}\lambda _{7}+\lambda _{2}\lambda _{6}\lambda
_{9})%
\end{array}%
\right\} .  \label{eqn3.43}
\end{equation}

Also,%
\begin{equation}
g_{M_{1}}(F,G)=\lambda _{1}\lambda _{4}+\lambda _{2}\lambda _{5}+\lambda
_{3}\lambda _{6},  \label{eqn3.44}
\end{equation}

and%
\begin{eqnarray}
&&Ric(F,G)  \label{eqn3.45} \\
&=&\lambda _{1}\lambda _{4}Ric(V,V)+\lambda _{1}\lambda
_{5}Ric(V,X_{1})+\lambda _{1}\lambda _{6}Ric(V,X_{2})  \notag \\
&&+\lambda _{2}\lambda _{4}Ric(X_{1},V)+\lambda _{2}\lambda
_{5}Ric(X_{1},X_{1})+\lambda _{2}\lambda _{6}Ric(X_{1},X_{2})  \notag \\
&&+\lambda _{3}\lambda _{4}Ric(X_{2},V)+\lambda _{3}\lambda
_{5}Ric(X_{2},X_{1})+\lambda _{3}\lambda _{6}Ric(X_{2},X_{2}).  \notag
\end{eqnarray}

Since $\dim (Ker\vartheta _{\ast })=1,$ so $Ric^{v}(V,V)=0$ and using (\ref%
{eqn3.39}), (\ref{eqn3.40}) and (\ref{eqn3.41}) in (\ref{eqn3.1}), we get%
\begin{equation}
Ric(V,V)=e^{2u_{1}}.  \label{eqn3.46}
\end{equation}

From (\ref{eqn3.39}) and (\ref{eqn3.40}) we have%
\begin{equation}
Ric(V,X_{1})=Ric(V,X_{2})=0.  \label{eqn3.47}
\end{equation}

From (\ref{eqn3.3}), we have%
\begin{equation}
Ric(X_{1},X_{1})=e^{2u_{1}}  \label{eqn3.48}
\end{equation}%
\begin{equation}
Ric(X_{2},X_{2})=e^{2u_{1}}.  \label{eqn3.49}
\end{equation}

Using (\ref{eqn3.46}), (\ref{eqn3.47}), (\ref{eqn3.48}) and (\ref{eqn3.49})
in (\ref{eqn3.45}), we have%
\begin{equation}
Ric(F,G)=e^{2u_{1}}\left\{ \lambda _{1}\lambda _{4}+\lambda _{2}\lambda
_{5}+\lambda _{3}\lambda _{6}\right\} .  \label{eqn3.50}
\end{equation}

Now, by using (\ref{eqn3.43}), (\ref{eqn3.44}) and (\ref{eqn3.50}) in (\ref%
{eqn3.42}), we have
\begin{eqnarray*}
&&\frac{1}{2}\left\{
\begin{array}{c}
e^{u_{1}}(\lambda _{1}\lambda _{5}\lambda _{7}-2\lambda _{1}\lambda
_{4}\lambda _{8}-2\lambda _{3}\lambda _{6}\lambda _{8} \\
+\lambda _{3}\lambda _{5}\lambda _{9}+\lambda _{2}\lambda _{4}\lambda
_{7}+\lambda _{2}\lambda _{6}\lambda _{9})%
\end{array}%
\right\} \\
&&+e^{2u_{1}}\left\{ \lambda _{1}\lambda _{4}+\lambda _{2}\lambda
_{5}+\lambda _{3}\lambda _{6}\right\} +\mu (\lambda _{1}\lambda _{4}+\lambda
_{2}\lambda _{5}+\lambda _{3}\lambda _{6}) \\
&=&0.
\end{eqnarray*}

So, $(M_{1},g_{M_{1}})$ admit a Ricci soliton for%
\begin{equation*}
\mu =-e^{2u_{1}}-\frac{e^{u_{1}}(\lambda _{1}\lambda _{5}\lambda
_{7}-2\lambda _{1}\lambda _{4}\lambda _{8}-2\lambda _{3}\lambda _{6}\lambda
_{8}+\lambda _{3}\lambda _{5}\lambda _{9}+\lambda _{2}\lambda _{4}\lambda
_{7}+\lambda _{2}\lambda _{6}\lambda _{9})}{\lambda _{1}\lambda _{4}+\lambda
_{2}\lambda _{5}+\lambda _{3}\lambda _{6}},
\end{equation*}

where $\lambda _{1}\lambda _{4}+\lambda _{2}\lambda _{5}+\lambda _{3}\lambda
_{6}\neq 0.$Ricci soliton $(M_{1},g_{M_{1}})$ becomes expanding, steady or
shrinking for some values of $\lambda _{i}$'s with respect to $\mu >0,$ $\mu
=0$ or $\mu <0,$ respectively.
\end{example}

M. Polat

Department of Mathematics, Faculty of Sci. Dicle University, {21280, Sur,
Diyarbak\i r-Turkey.}\newline
E-mail: murat.polat@dicle.edu.tr \{M. Polat\}

\end{document}